\documentclass[12pt]{amsart}
\usepackage{amsfonts,amscd,latexsym,amsmath,amssymb,enumerate,mathrsfs}
\theoremstyle{plain}

\makeatletter
\@namedef{subjclassname@2010}{%
  \textup{2010} Mathematics Subject Classification}
\makeatother

\newtheorem{master}{Master}[section]

\newtheorem{thm}[master]{Theorem}
\newtheorem{fact}[master]{Fact}
\newtheorem{lem}[master]{Lemma}

\newtheorem{question}[master]{Question}
\newtheorem{problem}[master]{Problem}

\theoremstyle{definition}
\newtheorem{defin}[master]{Definition}
\newtheorem{observation}[master]{Observation}
\theoremstyle{remark}
\newtheorem{remark}[master]{Remark}
\numberwithin{equation}{section}

\newcommand{\Rea}{\mathbb{R}}
\newcommand{\Nat}{\mathbb{N}}
\newcommand{\Int}{\mathbb{Z}}
\newcommand{\Rat}{\mathbb{Q}}

\begin{document}
\title[Non-abelian group structure on the Urysohn space]{Non-abelian group structure on the Urysohn universal space}
\author[M. Doucha]{Michal Doucha}
\address{Institute of Mathematics\\ Polish Academy of Sciences\\
00-656 Warszawa, Poland}
\email{m.doucha@post.cz}

\date{}
\begin{abstract}
Following the continuing interest in the Urysohn space and, more specifically, the recent problem area of finding and comparing group structures on the Urysohn space we prove that there exists a non-abelian group structure on the Urysohn universal metric space. More precisely, we introduce a variant of the Graev metric that enables us to construct a free group with countably many generators equipped with a two-sided invariant metric that is isometric to the rational Urysohn space. We provide several open questions and problems related to this subject.
\end{abstract}
\maketitle
\section*{Introduction}
There has been a lot of research recently connected to the Urysohn universal metric space. The space was constructed by P. Urysohn (\cite{Ur}) in 1920's but was forgotten for quite a long time. Nowadays, the Urysohn space, as well as the group of all its isometries, are a popular topic of mathematical research. A very interesting result was proved by P. Cameron and A. Vershik in \cite{CaVe} where they proved that there is an abelian (monothetic) group structure on the Urysohn space. Later, P. Niemiec in \cite{Nie1} proved that there is an abelian Boolean metric group that is isometric to the Urysohn space. And recently, Niemiec in \cite{Nie2} rediscovered the Shkarin's universal abelian Polish group (\cite{Sh}) and proved that this group is isometric to the Urysohn space as well (it is open though whether it differs from the group structures found by Cameron and Vershik). Niemiec also proved several negative results concerning group structures on the Urysohn space (we again refer to \cite{Nie2}), e.g. he proved there is no abelian metric group of exponent $3$ that is isometric to the Urysohn space (\cite{Nie2} Proposition 2.18). Let us also mention our previous work from \cite{Do} where we showed an existence of a metrically universal separable abelian metric group (answering an open question of Shkarin from \cite{Sh}) which turned out to be yet another different abelian group isometric to the Urysohn space. Vershik then asked (personal communication and \cite{Bonn}) whether there also exists a non-abelian group structure on the Urysohn space. We answer this question affirmatively here. Thus the following is the main result of this paper.
\begin{thm}\label{main}
There exists a free group $G$ of countably many generators equipped with a two-sided invariant metric that is isometric to the rational Urysohn space. In particular, there is a non-abelian group structure on the Urysohn space: the metric completion of $G$.
\end{thm}
The technical tool used for proving the theorem is an extension of two-sided invariant metric on a group to its free product with a free group. The referee of this paper pointed out to us the connection of this tool to the classical Graev metric introduced in \cite{Gr} and suggested to give it a more central role in the paper. We therefore state this tool in a general form, as suggested by the referee, as it might be of independent interest.
\begin{thm}\label{Graev}
Let $(G,d_G)$ be a group with two-sided invariant metric and let $(X,d_X)$ be a metric space. Suppose that $d'$ is a metric on the disjoint union $G\amalg X$ which extends both $d_G$ and $d_X$, and such that for every $x\in X$ we have $\inf \{d'(g,x):g\in G\}>0$ (equivalently, $G$ is closed in $G\amalg X$). Then $d'$ extends to the two-sided invariant metric $\delta$ on $G\ast F(X)$, where $F(X)$ is the free group with $X$ as a set of generators.
\end{thm}
\begin{remark}
If $G=\{1\}$, then $\delta$ from the statement of Theorem \ref{Graev} corresponds to the standard Graev metric (from \cite{Gr}) on $F(X)$, with $1$ as a unit, constructed over the pointed space $X\amalg\{1\}$.
\end{remark}
We also refer the reader to \cite{Slu} where a variant of Graev metric on free products of groups having a common closed subgroup was defined.

The subject of the group structures on the Urysohn space is still far from being finished and there are several open questions provided at the end of the paper. Since most of the groups isometric to the Urysohn space are constructed via Fra\" iss\' e theory, we provide few questions related to Fra\" iss\' e classes of metric groups as well at the end of the paper.
\section{Preliminaries and definitions}
Recall that the Urysohn universal metric space is a Polish metric space that contains an isometric copy of every finite metric space and every partial isometry between two finite subsets extends to an autoisometry of the whole space. These properties characterize the Urysohn space uniquely up to isometry and moreover imply that it contains an isometric copy of every separable metric space.

The rational Urysohn space is a countable metric space with all distances rational that contains an isometric copy of every finite rational metric space and every partial isometry between two finite subsets extends to an autoisometry of the whole space. Again, it follows that such a space is unique up to isometry and contains an isometric copy of every countable rational metric space. Moreover, one can prove that the completion of the rational Urysohn space is the Urysohn space.

Let us also recall that a function $f:X\rightarrow \Rea^+_0$ is called Kat\v etov, where $(X,d)$ is some metric space, if for every $x,y\in X$ we have $|f(x)-f(y)|\leq d(x,y)\leq f(x)+f(y)$. One should think about the Kat\v etov function $f$ as about a function that prescribes distances from some, potentially new, point. We refer the reader to \cite{Kat} for more information about Kat\v etov functions and the construction of the Urysohn space using them.

The following well known fact characterizes the Urysohn and the rational Urysohn spaces.
\begin{fact}\label{char_Urys}
\begin{enumerate}
\item Let $(X,d)$ be a countable metric space with rational metric. Then it is isometric to the rational Urysohn space iff for every finite subset $A\subseteq X$ and for every rational Kat\v etov function $f:A\rightarrow \mathbb{Q}^+$ there exists $x\in X$ such that $\forall a\in A (d(a,x)=f(a))$.
\item Let $(X,d)$ be a Polish metric space. Then it is isometric to the Urysohn space iff for every finite subset $A\subseteq X$ and for every Kat\v etov function $f:A\rightarrow \mathbb{R}^+$ there exists $x\in X$ such that $\forall a\in A (d(a,x)=f(a))$.
\end{enumerate}
\end{fact}
Later, when we construct the free group with two-sided invariant rational metric we check that it is isometric to the rational Urysohn space using the characterization from the previous fact.

We now define a special type of metric that is completely determined by its values on pairs from a finite set.
\begin{defin}
Let $(G,d)$ be a metric group. We say that the metric $d$ is \emph{finitely generated} if there exists a finite set $A_G\subseteq G$ (called generating set for $d$) such that $1\in A_G$, $A_G=A_G^{-1}$ and for every $a,b\in G$ we have $d(a,b)=\min\{d(a_1,b_1)+\ldots+d(a_n,b_n):n\in \Nat,\forall i\leq n (a_i,b_i\in A_G\wedge a=a_1\cdot\ldots\cdot a_n,b=b_1\cdot\ldots\cdot b_n)\}$. In particular, $G$ is (algebraically) generated by $A_G$.\\
\end{defin}
\begin{fact}\label{fact1}
If $d$ is a finitely generated metric on a group $G$, then $d$ is two-sided invariant.
\end{fact}
\begin{proof}
Recall (or it can be easily verified) that a metric $d$ on a group $G$ is two-sided invariant iff $\forall a,b,c,d\in G$ we have $d(a\cdot b,c\cdot d)\leq d(a,c)+d(b,d)$. It follows from the definition that finitely generated metrics have the latter property.
\end{proof}
Let us now present the observation of the referee which connects groups with finitely generated metric with free groups with Graev metric. Regarding Graev metric on a free group, we follow, and refer the reader to, Section 3 from \cite{DiGa}.
\begin{observation}
$G$ is a group with a finitely generated metric iff $G$ is a factor group of a free group of finitely many generators with the Graev metric with the factor metric.
\end{observation}
To see this, suppose that the finitely generated metric on $G$ is generated by a finite set $A_G\subseteq G$. Consider the free group $F(A_G\setminus \{1\})$, with $A_G\setminus \{1\}$ as a set of free generators and $1\in A_G$ as a unit, with the Graev metric. Since $A_G$ also (algebraically) generates $G$ there is a natural homomorphism from $F(A_G\setminus \{1\})$ onto $G$. It follows from the definitions of the respective metrics that this homomorphism is $1$-Lipschitz and that the distance between two elements of $G$ is equal to the infimum distance between the corresponding classes in $F(A_G\setminus \{1\})$.

Conversely, if $F(X)$ is a free group with the Graev metric constructed over a finite pointed metric space $X\amalg \{1\}$, $H\leq F(X)$ is a closed subgroup, then the factor metric on $F(X)/H$ is finitely generated. One can check that the generating set for the factor metric is $\{[a]_H: a\in X\amalg X^{-1}\amalg \{1\}\}$.
\section{Proofs of the main theorems}
Having defined finitely generated metric, we restate Theorem \ref{Graev} here adding a special subcase when $d_G$ is finitely generated and $X$ is finite.
\begin{thm}\label{reGraev}
Let $(G,d_G)$ be a group with two-sided invariant metric and let $(X,d_X)$ be a metric space.
\begin{enumerate}
\item  Suppose that $d'$ is a metric on the disjoint union $G\amalg X$ which extends both $d_G$ and $d_X$, and such that for every $x\in X$ we have $\inf \{d'(g,x):g\in G\}>0$ (equivalently, $G$ is closed in $G\amalg X$). Then $d'$ extends to the two-sided invariant metric $\delta$ on $G\ast F(X)$, where $F(X)$ is the free group with $X$ as a set of generators.
\item If $d_G$ is finitely generated by values on pairs of some finite $A_G\subseteq G$, $X$ is finite and $d'$ is a metric on the disjoint union $A_G\amalg X$ which extends both $d_G\upharpoonright A_G$ and $d_X$, then $d'$ extends to the finitely generated metric $\delta$ on $G\ast F(X)$ such that $\delta\upharpoonright G=d_G$.

\end{enumerate}
\end{thm}
At first, we show how to deduce Theorem \ref{main} from Theorem \ref{reGraev}. For every $m\in \Nat$, by $F_m$ we shall denote the free group of $m$ generators.

Suppose that $f$ is a rational Kat\v etov function defined on some finite set $B$, where $B\subseteq F_m$ and $F_m$ is equipped with a finitely generated rational metric. Call such a Kat\v etov function \emph{relevant}. Since there are only countably many finitely generated rational metrics on free groups of finitely many generators, it follows that there are only countably many relevant Kat\v etov functions.  So let $(f_n)_{n\in \mathbb{N}}$ be an enumeration of all relevant Kat\v etov functions with infinite repetition.

We construct the group $G$ inductively as a direct limit of free groups of finitely many generators equipped with a (two-sided invariant) finitely generated rational metric which is a variant of the Graev metric (defined below). A construction using a direct limit of groups equipped with the Graev metric was first used in \cite{DrGa} to produce a Polish group in which Lie sums and Lie brackets do not exist.

At step $1$, we set $F_1\cong \Int$ to be the integers with the standard Euclidean metric $d_1$, which is clearly finitely generated and rational.

Suppose we have constructed the free group $F_n$ with finitely generated rational metric $d_n$ which is generated by values on pairs of some finite set $A_n\subseteq F_n$. Consider the relevant Kat\v etov function $f_n$. We have that $f_n:B\subseteq (F_m,p)\rightarrow \Rat^+$, for some $m\in \Int$, finitely generated rational metric $p$ and finite set $B$. Suppose that $m\leq n$ and $(F_m,p)$ is isometrically isomorphic to $(F_m,d_n\upharpoonright F_m)$, where $F_m$ is naturally identified with the free subgroup of $F_n$ generated by the first $m$ generators. Then we can actually view the function $f_n$ as defined on some finite set $B\subseteq F_n$. Without loss of generality, we may suppose that $A_n$, the generating set for $d_n$, is equal to $B$. Indeed, we could extend $f_n$ to $A_n\cup B$ and then $A_n\cup B$ would still be generating set for $d_n$.

Using $(2)$ of Theorem \ref{reGraev} with $X$ as a one-point space $\{x\}$ and $d'(x,g)=f(g)$, for every $g\in B=A_n$, we extend the metric $d_n$ to a (finitely generated rational) metric $d_{n+1}$ on $F_n\ast F_1\cong F_{n+1}$ such that the Kat\v etov function $f_n$ is realized by the new added generator.

If, on the other hand, we have that either $m>n$ or $m\leq n$ but  $(F_m,p)$ is \emph{not} isometrically isomorphic to $(F_m,d_n\upharpoonright F_m)$, then we extend $(F_n,d_n)$ to $(F_{n+1},d_{n+1})$ arbitrarily (just ensuring that $d_{n+1}$ is still finitely generated and rational).\\

When the inductive construction is finished, we have a free group with countably many generators, denoted by $G$, equipped with some two-sided invariant rational metric $d$. It follows that the group operations on $G$ are continuous with respect to the topology induced by the metric. To see this, just observe that by invariance for any $g,h\in G$ we have 
\begin{equation}\label{inverse}
d(g,h)=d(g^{-1},h^{-1})
\end{equation}
so the operation inverse is continuous (an isometry), and for any $g_1,g_2,\\h_1,h_2\in G$ we have 
\begin{equation}\label{addition}
d(g_1\cdot h_1,g_2\cdot h_2)\leq d(g_1,g_2)+d(h_1,h_2)
\end{equation}
by invariance and triangle inequality.

Consider now the metric completion, denoted $\mathbb{G}$, of $G$. It is a separable complete metric space and the group operations extend to the completion. Indeed, the inverse operation extends because it is an isometry (\ref{inverse}) and the group multiplication extends because if $(g_n)_n,(h_n)_n\subseteq G$ are two Cauchy sequences then $(g_n\cdot h_n)_n$ is a Cauchy sequence as well (\ref{addition}). It follows that $\mathbb{G}$ is a Polish group equipped with a two-sided invariant metric. We refer the reader to \cite{Gao} for an exposition on Polish (metric) groups.\\

We claim that $G$ is isometric to the rational Urysohn space. It suffices to check the condition from Fact \ref{char_Urys} (1). So let $f:A\subseteq G\rightarrow \Rat^+$ be an arbitrary rational Kat\v etov function defined on a finite subset $A$. Then there exists some $m\in \Int$ such that $A\subseteq F_m$ and there are infinitely many $n$'s such that $f$ corresponds to $f_n$. Choose one such $n$ that is greater than $m$. However, then during the $n$-th step of the induction we guaranteed that $f_n$, and thus $f$, was realized in $F_{n+1}$.

Thus the rest of the section is devoted to prove Theorem \ref{reGraev}. We first prove $(1)$ of Theorem \ref{reGraev} and then show how the item $(2)$ follows.\\

The reader is invited to compare the tools of the proof with those in Section 3 of \cite{DiGa}. Let $X^{-1}$ be a disjoint copy of $X$, considered as a set of formal inverses of elements of $X$. For every $x\in X^i$, $i\in \{-1,1\}$, $x^{-1}$ denotes the corresponding element in $X^{-i}$. We extend $d'$ on $G\amalg X\amalg X^{-1}$, denoted by $d$,  so that:
\begin{itemize}
\item For every $a,b\in G\amalg X^{-1}$ we have $d(a,b)=d'(a^{-1},b^{-1})$.
\item For every $a\in G\amalg X$ and $b\in G\amalg X^{-1}$ we have $d(a,b)=\inf \{d'(a,c)+d'(c^{-1},b^{-1}):c\in G\}$.

\end{itemize}
In other words, at first we define the distances between elements of $G\amalg X^{-1}$ so that the bijection between $G\amalg X$ and $G\amalg X^{-1}$ that takes $a\to a^{-1}$, for every $a\in G\amalg X$, is an isometry. Then we take the (greatest) metric amalgamation of $G\amalg X$ and $G\amalg X^{-1}$ over $G$.

Denote now $G\amalg X\amalg X^{-1}$ by $S$, and let $W(S)$ be the set of all words over $S$ considered as an alphabet. 
\begin{defin}\label{irrdef}
A word $w=w_1\ldots w_n\in W(S)$, where $w_i\in S$ for $i\leq n$, is called \emph{irreducible} if for no $i<n$ we have $w_i,w_{i+1}\in G$ or $w_i=w_{i+1}^{-1}$.
\end{defin}
For every $w\in W(S)$ we shall denote by $w'$ the corresponding element in $G\ast F(X)$. Note that the mapping $w\to w'$ is a bijection between irreducible words from $W(S)$ and elements of $G\ast F(X)$.

For any $w\in W(S)$, let $|w|$ denote its length. If $v,w\in W(S)$ are two words of the same length $n$, then we define the pre-distance between them as $$\rho(v,w)=d(v_1,w_1)+\ldots+d(v_n,w_n)$$
Finally, we define the \emph{Graev} metric $\delta$ on $G\ast F(X)$ as
$$\delta(u,v)=\inf \{\rho(u^*,v^*):u^*,v^*\in W(S),|u^*|=|v^*|, (u^*)'=u,(v^*)'=v\}$$ for any $u,v\in G\ast F(X)$. It is easy to check that $\delta$ is symmetric and that for any $u,v,w,x\in G\ast F(X)$ we have $\delta(u\cdot v,w\cdot x)\leq \delta(u,w)+\delta(v,x)$. The latter property also implies two-sided invariance and the triangle inequality. We need to check that it is indeed a metric, i.e. $\delta(u,v)>0$ if $u\neq v$, and that it extends $d'$.

We need the following definition. Compare it with Definition 3.3 in \cite{DiGa}.
\begin{defin}[Match]\label{definematch}
Let $w\in W(S)$ be a word of length $n$. Let $P\subseteq \{1,\ldots,n\}$ be the subset such that for every $i\leq n$ we have $i\in P$ iff $w_i\in X\amalg X^{-1}$. We call a function $\theta:P\rightarrow P$ a match for $w$ if
\begin{itemize}
\item for every $i\in P$, we have $\theta(\theta(i))=i$
\item for every $i\in P$, we have $w_i=w_{\theta(i)}^{-1}$
\item for every $i\in P$, assuming without loss of generality that $i<\theta(i)$, we have $$\prod_{i\leq j\leq \theta(i)} w'_j=1$$
\end{itemize}
\end{defin}
\begin{lem}\label{lem1}
Let $w\in W(S)$ be such that $w'\in G$. Then there exists a match for $w$.
\end{lem}
\emph{Proof of Lemma \ref{lem1}}. Let $n=|w|$. First of all, we suppose that there exists $i\leq n$ such that $w_i\in X\amalg X^{-1}$. Otherwise, there is nothing to prove.

Next we claim that it suffices to prove the lemma for the case when $w_1\in X\amalg X^{-1}$ and $w_n=w_1^{-1}$, $w'=1$ and for no $i<n$ we have $(w_1\ldots w_i)'=1$. Let us call such a sequence $w_1,\ldots,w_n$ a cancelling $X$-sequence of length $|P|$, where again $P=\{i\leq n: w_i\in X\amalg X^{-1}\}$. Indeed, suppose the lemma is proved for such a case. Consider the first index $1\leq i_S<n$ such that $w_{i_S}\in X\amalg X^{-1}$. Since $w'\in G$ there must exist index $i_S<i\leq n$ such that $(w_{i_S}\ldots w_i)'=1$. Let $i_F$ be the least such index. Clearly, $w_{i_S}=w_{i_F}^{-1}$. By assumption, we can find an appropriate match $\theta_1$ for the subword $w_{i_S}\ldots w_{i_F}$. Then we look for the least index $i_F<i'_S$, if there is any, such that $w_{i'_S}\in X\amalg X^{-1}$. Again, we can find an appropriate $i'_S<i'_F\leq n$ and then find a mathc $\theta_2$ for the subword $w_{i'_S}\ldots w_{i'_F}$. At the end, we can take as the desired $\theta$ the union $\theta_1\cup\theta_2\cup\ldots$ of all matches obtained in that way.

We now prove the lemma (with the assumption that $w_1,\ldots,w_n$ is a cancelling $X$-sequence) by induction on $|P|$. If $|P|=2$ then, clearly, we may put $\theta(1)=n$ and $\theta(n)=1$ and we are done.

Suppose now that $|P|=m>2$ and the lemma has been proved for all (even) $l<m$. Suppose that $P=\{k_1=1,\ldots,k_m=n\}\subseteq \{1,\ldots,n\}$. Since $(w_1\ldots w_n)'=1$, $w_1=w_n^{-1}$, we have $(w_2\ldots w_{n-1})'=1$. Thus there must exist $2<i<n$ such that $(w_{k_2}\ldots w_{k_i})'=1$. Suppose also that $i$ is the least index with such a property. Then both $w_{k_2},\ldots,w_{k_i}$ and $w_1,\ldots,w_{k_2-1},\\w_{k_i+1},\ldots,w_n$ are cancelling $X$-sequences of length less than $m$. By induction hypothesis, we can find corresponding $\theta_1:\{k_2,\ldots,k_i\}\rightarrow \{k_2,\ldots,k_i\}$ and $\theta_2: \{k_1,k_{i+1},\ldots,k_m\}\rightarrow \{k_1,k_{i+1},\ldots,k_m\}$. We then set $\theta=\theta_1\cup \theta_2$ and we are done.
\hfill $\qed$ (of Lemma \ref{lem1})\\
\begin{lem}\label{lem2}
Let $g\in G\ast F(X)$. Then $\delta(g,1)=\inf \{\rho(w,u):w,u\in W(S),|w|=|u|,w'=g,u'=1,w\text{ is irreducible}\}$.
\end{lem}
\emph{Proof of Lemma \ref{lem2}.} Let $w,u\in W(S)$ be such that $|w|=|u|=n$ and $w'=g$, $u'=1$. Suppose that $w$ is not irreducible. We show that we can then reduce the words $w$ and $u$ to the words $\bar{w},\bar{u}$ such that we still have $|\bar{w}|=|\bar{u}|$, $\bar{w}'=g,\bar{u}'=1$, however, $\rho(\bar{w},\bar{u})\leq \rho(w,u)$.

Let $\theta: P\rightarrow P$ be a match for $u$. Since $w$ is not irreducible, according to Definition \ref{irrdef} there is either $i<n$ such that $w_i,w_{i+1}\in G$ or $w_i,w_{i+1}\in X\amalg X^{-1}$ and $w_i=w_{i+1}^{-1}$. We shall treat these cases separately.
\begin{itemize}
\item {\bf Case 1}. Suppose that $w_i,w_{i+1}\in G$.\\

{\it Subcase \bf 1a}. If $u_i,u_{i+1}\in G$ as well, then we could reduce $w$ and $u$ to $\bar{w},\bar{u}$ so that for every $j<i$, $\bar{v}_j=v_j$, $\bar{v}_i=v_i\cdot v_{i+1}$ and for every $i<j<n$, $\bar{v}_j=v_{j+1}$, where $v$ is either $w$ or $u$. In that case, we have $\rho(\bar{w},\bar{u})\leq \rho(w,u)$ since $d(w_i'\cdot w'_{i+1},u'_i\cdot u'_{i+1})\leq d(w'_i,u'_i)+d(w'_{i+1},u'_{i+1})$ by two-sided invariance of $d$ on $G$.\\

{\it Subcase \bf 1b}. So suppose that either $u_i$ or $u_{i+1}$ belong to $X\amalg X^{-1}$. Let us say that $u_i\in X\amalg X^{-1}$. We shall find $\tilde{u}\in W(S)$ such that $|\tilde{u}|=|u|$, $\tilde{u}'=1$, $\rho(w,\tilde{u})\leq \rho(w,u)$ and $\tilde{u}_i\in G$. Let $j=\theta(i)$ and suppose that $j>i$, the other case is analogous. Then since $\theta$ is a match for $u$, we have $u_j=u_i^{-1}$. Thus we have $d(w_j,u_j)+d(u_i,w_i)\geq d(w_j^{-1},w_i)$. Since we have $\prod_{i<k<j} u_k'=1$ we can modify $u$ to $\tilde{u}$ so that $\tilde{u}_i=w_j^{-1}$ and $\tilde{u}_j=w_j$ and for $k\in \{1,\ldots,n\}\setminus \{i,j\}$ we have $\tilde{u}_k=u_k$. Then $\tilde{u}$ is as required since $\rho(w,u)-\rho(w,\tilde{u})=(d(w_j,u_j)+d(u_i,w_i))-(d(w_j^{-1},w_i)+d(w_j,w_j))\geq 0$. If $\tilde{u}_{i+1}\in G$, then we are in Subcase 1a. Otherwise, apply the procedure above also for $\tilde{u}_{i+1}$. Then we will be in Subcase 1a.\\
\item {\bf Case 2}. Suppose that $w_i,w_{i+1}\in X\amalg X^{-1}$ and $w_{i+1}=w_i^{-1}$.\\

{\it Subcase \bf 2a}. Suppose that either $u_i$ or $u_{i+1}\in G$. Let us say $u_i\in G$. Then we have $d(u_i,w_i)+d(w_{i+1},u_{i+1})\geq d(u_i,u_{i+1}^{-1})=d(u_i^{-1},u_{i+1})$. Then we can replace $w_i$ by $u_i$ and $w_{i+1}$ by $u_i^{-1}$ in $w$ to obtain $\tilde{w}$. Again clearly, $\rho(\tilde{w},u)\leq \rho(w,u)$. Note that both $\tilde{w}_i$ and $\tilde{w}_{i+1}$ are then in $G$. Thus we are in Case 1.\\

{\it Subcase \bf 2b}. Suppose that both $u_i$ and $u_{i+1}$ are in $X\amalg X^{-1}$. Using the match $\theta$ and arguing as in Subcase 1a, we can check that $d(w_i,u_i)+d(u_{\theta(i)},w_{\theta(i)})\geq d(w_i^{-1},w_{\theta(i)})$ and that  $d(w_{i+1},u_{i+1})+d(u_{\theta(i+1)},w_{\theta(i+1)})\geq d(w_{i+1}^{-1},w_{\theta(i+1)})$. It follows that we may modify $u$ to $\tilde{u}$ so that $\tilde{u}_i=w_i$, $\tilde{u}_{i+1}=w_{i+1}$ and $\tilde{u}_{\theta(i)}=w_i^{-1}$, $\tilde{u}_{\theta(i+1)}=w_{i+1}^{-1}$; at other positions, $\tilde{u}$ is equal to $u$. It follows that $\rho(w,\tilde{u})\leq \rho(w,u)$ and we can erase $w_i=\tilde{u}_i=w_{i+1}^{-1}=\tilde{u}_{i+1}^{-1}$ from $w$ and $\tilde{u}$ respectively.
\end{itemize}
\hfill $\qed$ (of Lemma \ref{lem2})\\

We are ready to finish the proof of Theorem \ref{reGraev}, $(1)$. Recall that it remains to prove that $\delta(x,y)>0$ if $x\neq y$, and that $\delta$ extends $d$.\\

For the former, since $\delta$ is two-sided invariant it suffices to check that for any $x\in G\ast F(S)$ such that $x\neq 1$ we have $\delta(x,1)>0$. Let $w\in W(S)$ be the irreducible word such that $w'=x$. Let $n=|w|$. By Lemma \ref{lem2}, we have $\delta(x,1)=\inf\{\rho(w,u):u\in W(S),|u|=n,u'=1\}$. By assumption $$\varepsilon_0=\min_{\{w_i\in X\amalg X^{-1}\}}\inf\{d(g,w_i):g\in G\}>0$$
Let $u\in W(S)$ be arbitrary such that $|u|=n$ and $u'=1$. If there exists $i\leq n$ such that $w_i\in X\amalg X^{-1}\wedge u_i\in G$, then $\rho(w,u)\geq \varepsilon_0>0$. Suppose there exists $i\leq n$ such that $w_i\in G\wedge u_i\in X\amalg X^{-1}$. Let $\theta:P\rightarrow P$ be a match for $u$, where again $P=\{i\leq n:u_i\in X\amalg X^{-1}\}$. Then we could replace $u$ by $u^*$ such that $u^*_j=u_j$ for $j\in \{1,\ldots,n\}\setminus \{i,\theta(i)\}$, and $u^*_i=w_i$ and $u^*_{\theta(i)}=w_i^{-1}$. Indeed, since $d(w_{\theta(i)},w_i^{-1})\leq d(w_i,u_i)+d(u_i^{-1},w_{\theta(i)})$, it follows that $\rho(w,u^*)\leq \rho(w,u)$.

Consequently, we may suppose that for every $i\leq n$ we have $w_i\in G$ iff $u_i\in G$. Indeed, if for some $i\leq n$ we have $w_i\in X\amalg X^{-1}$ and $u_i\in G$, then we argued above that then we have $\rho(w,u)\geq \varepsilon_0>0$. If, on the other hand, for some $i\leq n$ we have $w_i\in G$ and $u_i\in X\amalg X^{-1}$, then we argued above that we can replace $u$ by $u^*$ such that $(u^*)'=1$, $|u^*|=|u|$, $u^*_i\in G$ and  $\rho(w,u^*)\leq \rho(w,u)$.

If for every $i\leq n$ we have $w_i\in G$, then since $w$ is irreducible we have $w=w_1$, $u=u_1=1$ and clearly, $\rho(w,u)=d(w,1)>1$. Thus we suppose that $P\neq \emptyset$. Let $$\varepsilon_1=\min \{d(w_i,w_j^{-1}):i,j\in P,w_i\neq w_j^{-1}\}$$ and let $$\varepsilon_2=\min\{d(w_j,1):w_j\in G\}$$ If there exists $i\in P$ such that $w_i\neq u_i$ then we have $\rho(w,u)\geq d(w_i,u_i)+d(w_{\theta(i)},u_{\theta(i)})=d(w_i,u_i)+d(w_{\theta(i)},u_i^{-1})\geq d(w_i,w_{\theta(i)}^{-1})\geq \varepsilon_1>0$, since $u_i=u_{\theta(i)}^{-1}$.

Otherwise, for every $i\in P$ we have that $w_i=u_i$. However, since $w'\neq 1$, there exists $i\in P$ such that $\theta(i)>i$ and for every $i<j<\theta(i)$ we have $j\notin P$. We claim that either $\theta(i)=i+1$ or $\theta(i)=i+2$. Indeed, if $\theta(i)>i+1$, then $w_{i+1}\in G$, and since $w$ is irreducible we must have $w_{i+2}\in X\amalg X^{-1}$, so $i+2\in P$ and the claim follows. If the first case holds, i.e. $\theta(i)=i+1$, we have $\rho(w,u)\geq d(w_i,u_i)+d(w_{i+1},u_i^{-1})\geq d(w_i,w_{i+1}^{-1})\geq \varepsilon_1>0$. If the other case holds, i.e. $\theta(i)=i+2$, then by the definition of match we must have $u_{i+1}=1$, and thus $\rho(w,u)\geq d(w_{i+1},1)\geq \varepsilon_2>0$.\\

For the latter, let $x,y\in S$. We need to check that $\delta(x,y)=d(x,y)$. Clearly, $\delta(x,y)\leq d(x,y)$. By two-sided invariance of $\delta$ and Lemma \ref{lem2} we have $\delta(x,y)=\delta(x\cdot y^{-1},1)=\inf\{d(x,z)+d(y^{-1},z^{-1}):z\in S\}$. However, the infimum is attained for $z=x$ or $z=y$ since $d(x,z)+d(y^{-1},z^{-1})=d(x,z)+d(z,y)\geq d(x,y)$, and we are done.\\

It remains to prove the item $(2)$ of Theorem \ref{reGraev}. First of all, we consider the greatest metric $d''$ on $G\amalg X$ that extends $d'$ on $A_G\amalg X$ and $d_G$ on $G$. More precisely, $d''$ is the amalgam metric of $d'$ on $A_G\amalg X$ and $d_G$ on $G$ over $A_G$, i.e. for any $x\in X$ and $g\in G$, we have $d''(x,g)=\inf\{d'(x,g_0)+d_G(g_0,g):g_0\in A_G\}$. Observe that the infimum is in fact attained since $A_G$ is finite. Thus in particular, if $d'$ is rational, $d''$ will be rational as well. Next, we extend $d''$ to $\delta$ as in the item $(1)$. We need to verify that $\delta$ is in that case finitely generated, and if $d'$ is rational, then $\delta$ is also still rational.

Let $d$ be the extension of $d''$ to $G\amalg X\amalg X^{-1}$ as in the proof of the previous item. Define a metric $\gamma$ on $G\ast F(X)$ as follows: for any $x,y\in G\ast F(X)$ we set $$\gamma(x,y)=\inf \{d(x_1,y_1)+\ldots+d(x_n,y_n):n\in \Nat,x_1,y_1,\ldots,x_n,$$ $$y_n\in A_G\amalg X\amalg X^{-1},x=x_1\cdot\ldots\cdot x_n,y=y_1\cdot\ldots\cdot y_n\}$$ It follows from the definition that $\gamma$ is a two-sided invariant metric which is finitely generated by the values on pairs from $A_G\amalg X\amalg X^{-1}$. Moreover, the minimum in the previous definition is actually attained since $A_G\amalg X\amalg X^{-1}$ is finite. Thus, if $d'$ is rational, then $\gamma$ is also rational. Comparing the definitions of $\gamma$ and $\delta$ in this particular case, one can see that they are equal. Indeed, for any $x,y\in G\ast F(X)$ we have $$\delta(x,y)=\inf \{\rho(w_x,w_y):w_x,w_y\in W(S),w'_x=x,w'_y=y,|w_x|=|w_y|\}=$$ $$\inf\{d(x_1,y_1)+\ldots+d(x_m,y_m):x_1,y_1,\ldots,x_m,y_m\in G\amalg X\amalg X^{-1},$$ $$x=x_1\cdot\ldots\cdot x_m,y=y_1\cdot\ldots\cdot y_m\}$$ Observe that in the previous equivalent definition of $\delta$, the elements\\ $x_1,y_1,\ldots, x_m,y_m$ are allowed to be from $G\amalg X\amalg X^{-1}$, while in the definition of $\gamma$ they have to be from $A_G\amalg X\amalg X^{-1}$. Thus, $\delta(x,y)\leq \gamma(x,y)$. However, for every $f,g\in G$ we have $d(f,g)=d(f_1,g_1)+\ldots+d(f_j,g_j)$ for some $f_1,g_1,\ldots,f_j,g_j\in A_G$ since $d\upharpoonright G=d_G$ is finitely generated by $A_G$. Similarly, for every $z\in X\amalg X^{-1}$ and $g\in G$ we have $d(z,g)=d(z,h)+d(h,g)$ for some $h\in A_G$. Thus for some $h_1,g_1,\ldots,h_l,g_l\in A_G$ we have $d(h,g)=d(h_1,g_1)+\ldots+d(h_l,g_l)$, so $d(z,g)=d(z,h)+d(h^{-1},h^{-1})+d(h_1,g_1)+\ldots+d(h_l,g_l)$. So it follows that actually $$\delta(x,y)=\inf\{d(x_1,y_1)+\ldots+d(x_m,y_m):x_1,y_1,\ldots,x_m,y_m\in A_G\amalg X\amalg X^{-1},$$ $$x=x_1\cdot\ldots\cdot x_m,y=y_1\cdot\ldots\cdot y_m\}=\gamma(x,y)$$
\section{Open questions and problems}
\subsection{Groups isometric to the Urysohn space}
To summarize, there are now five known group structures on the Urysohn space \footnote{One should rather talk about classes of group structures since Cameron-Vershik's example is a class of continuum many different monothetic group structures on the Urysohn space}, the groups from papers \cite{CaVe}, \cite{Nie1}, \cite{Nie2} (and \cite{Sh}), \cite{Do} and the present paper. Four of them are known to be different, it is open whether Shkarin/Niemiec's group belongs to the Cameron-Vershik's class. We provide some open questions from this area.

Let us start with the groups of finite exponent. We already mentioned in the introduction that Niemiec in \cite{Nie1} proved that there is an abelian metric group of exponent $2$ isometric to the Urysohn space and that he proved in \cite{Nie2} that there is no abelian metric group of exponent $3$ isometric to the Urysohn space. Moreover, consider the Fra\" iss\' e class of all finite abelian groups of exponent $n$, where $n>3$, equipped with invariant rational metric. He showed (Theorem 5.5 in \cite{Nie2}) that, surprisingly, the corresponding Fra\" iss\' e limit is not isometric to the rational Urysohn space. However, the following problem is still open.
\begin{question}[Niemiec]
Does there exist an abelian metric group of finite exponent other than $2$ and $3$ that is isometric to the Urysohn space?
\end{question}
Since all known metric groups isometric to the Urysohn space have an invariant metric and a countable dense subgroup isometric to the rational Urysohn space, it is probably worthy to work on the following problem.
\begin{problem}
Characterize countable groups that admit a \underline{two-sided} invariant metric with which they are isometric to the rational Urysohn space.
\end{problem}
The reason why we stressed that the metric should be two-sided invariant is because in such a case the group operations are automatically continuous and the operations extend to the metric completion. The following question is thus natural in this context.
\begin{question}
Does there exist a metric group that is isometric to the (rational) Urysohn space such that its metric is not two-sided invariant?
\end{question}
\subsection{Fra\" iss\' e classes of metric groups}
The natural class of all finite abelian groups equipped with invariant rational metric is rather easily checked to be a Fra\" iss\' e class and the metric completion of the corresponding Fra\" iss\' e limit is the universal Polish abelian group from papers \cite{Sh} and \cite{Nie2}. However, the analogous problem for the non-abelian case is open.
\begin{question}\label{non-abelianHall}
Does the class of all finite groups equipped with two-sided invariant rational metric have the amalgamation property?
\end{question}
Let us note that the class of all finite groups does have the amalgamation property (\cite{Neu}) and the Fra\" iss\' e limit is the Hall's universal locally finite group (\cite{Ha}). It is not hard to check that if the class from Question \ref{non-abelianHall} were Fra\" iss\' e, then the Fra\" iss\' e limit would be algebraically isomorphic to the Hall's group. It is not clear though whether it would be isometric to the rational Urysohn space.

\bigskip

\subsection*{Acknowledgements}
The author is very grateful to the referee of this paper for the excellent report where it was pointed out to us the connection between our original construction and the Graev metric and suggested a generalization which is presented in this final form.

The author is also grateful to Wies\l aw Kubi\' s for discussions on this topic.

Part of this work was done during the trimester program ``Universality and Homogeneity" at the Hausdorff Research Institute for
Mathematics in Bonn. The author would therefore like to thank for the support and great working conditions there.

The author was supported by funds allocated to the implementation of the international co-funded project in the years 2014-2018, 3038/7.PR/2014/2, and by the EU grant PCOFUND-GA-2012-600415.

\end{document}